\documentclass[12pt]{amsart}

\usepackage[pdfauthor   = {Sebastian\ Posur},
            pdftitle    = {Localization},
            pdfsubject  = {},
            pdfkeywords = {computable ring;\ coherent strongly discrete ring;\ linear system;\ localization},
            bookmarks=true,
            bookmarksopen=true,
            pagebackref=true,
            hyperindex=true,
            colorlinks=true,
            linkcolor=blue,
            citecolor=blue,
            filecolor=blue,
            urlcolor=blue,
            ]{hyperref}

\usepackage[utf8]{inputenc} 
\usepackage[T1]{fontenc}
\usepackage{a4wide}
\usepackage{lmodern}
\usepackage[english]{babel}
\usepackage{mathrsfs}
\usepackage{etex}
\usepackage{mathtools}
\usepackage{latexsym}
\usepackage{amssymb}
\usepackage{amsthm}
\usepackage{amsmath}
\usepackage{caption}

\usepackage{colortbl}

\usepackage[all]{xy}
\usepackage{verbatim}
\usepackage{listings}
\usepackage{fancyvrb}

\usepackage{graphicx}

\usepackage[dvipsnames]{xcolor}
\usepackage{accents} 
\usepackage{enumerate}
\usepackage{wrapfig}
\usepackage{tikz}
\usepackage{tikz-cd}
\usetikzlibrary{automata,shapes,arrows,matrix,backgrounds,positioning,plotmarks,calc,patterns,matrix,decorations.pathreplacing,decorations.pathmorphing,decorations.text,decorations.markings}
\usepackage[colorinlistoftodos,shadow]{todonotes}

\usepackage{multirow}
\usepackage{mdwlist}

\usepackage{stmaryrd}
\usepackage{mathdots} 

\usepackage{fancyvrb}

\usepackage[toc,page]{appendix}
\usepackage{float}

\usepackage{extarrows}
\usepackage{pdflscape}
\usepackage{rotating}

\newtheoremstyle{mytheoremstyle} 
    {5pt}                    
    {5pt}                    
    {\itshape}                   
    {\parindent}                           
    {\bf}                   
    {.}                          
    {.5em}                       
    {}  

\theoremstyle{mytheoremstyle}

\newtheorem{theorem}{Theorem}[section]

\newtheorem{lemma}[theorem]{Lemma}
\newtheorem{proposition}[theorem]{Proposition}
\newtheorem{corollary}[theorem]{Corollary}

\newtheoremstyle{mytdefintionstyle} 
    {5pt}                    
    {5pt}                    
    {\rm}                   
    {\parindent}                           
    {\bf}                   
    {.}                          
    {.5em}                       
    {}  

\theoremstyle{remark}
\newtheorem{remark}[theorem]{Remark}

\theoremstyle{mytdefintionstyle}
\newtheorem{definition}[theorem]{Definition}

\newtheorem{example}[theorem]{Example}

\newtheorem{construction}[theorem]{Construction}

\newtheoremstyle{exmp_contd} 
{\topsep} {\topsep}%
{\upshape}
{}
{\bfseries}
{}
{ }
{\thmname{#1}\,\thmnumber{ #2}\thmnote{#3}\enspace(continued)}

\theoremstyle{exmp_contd}

\usepackage{xspace}
\definecolor{ExQ}{HTML}{0000FF}
\definecolor{Dec}{HTML}{E07B00}

\newcommand{\CapPkg}{\textsc{Cap}\xspace}

\newcommand{\pmatrow}[2]{ \begin{pmatrix}{#1} & {#2} \end{pmatrix} }
\newcommand{\pmatcol}[2]{ \begin{pmatrix}{#1} \\ {#2} \end{pmatrix} }

\newcommand{\LT}{\mathrm{LT}}

\newcommand{\N}{\mathbb{N}}
\newcommand{\Z}{\mathbb{Z}}

\newcommand{\tr}{\mathrm{tr}}

\DeclareMathOperator{\cokernel}{\mathrm{coker}}

\DeclareMathOperator{\domain}{\mathrm{ann}}

\newcommand{\fpmod}{\text{-}\mathrm{fpmod}}

\tikzset{round left paren/.style={ncbar=0.5cm,out=120,in=-120}}
\tikzset{round right paren/.style={ncbar=0.5cm,out=60,in=-60}}
\newcolumntype{C}[1]{>{\centering\arraybackslash$}p{#1}<{$}}
\newlength{\mycolwd}
\usepackage{array, xcolor}
\definecolor{lightgray}{gray}{0.8}
\newcolumntype{L}{>{\raggedleft}p{0.28\textwidth}}
\newcolumntype{R}{p{0.8\textwidth}}

\definecolor{ctcolor}{gray}{0.95}

\definecolor{ctucolor}{gray}{0.85}

\makeatletter
\newcommand{\thickhline}{%
    \noalign {\ifnum 0=`}\fi \hrule height 1pt
    \futurelet \reserved@a \@xhline
}
\newcolumntype{"}{@{\hskip\tabcolsep\vrule width 1pt\hskip\tabcolsep}}
\makeatother

\usepackage{etex}

\author{Sebastian Posur}
\thanks{The author is supported by Deutsche Forschungsgemeinschaft (DFG) grant SFB-TRR 195: \emph{Symbolic Tools in Mathematics and their Application}}
\address{Department of mathematics, University of Siegen, 57068 Siegen, Germany}
\email{\href{mailto:Sebastian Posur <sebastian.posur@uni-siegen.de>}{sebastian.posur@uni-siegen.de}}

\begin{document}

\title[Linear systems over localizations of rings]{Linear systems over localizations of rings}

\begin{abstract}
We describe a method for solving linear systems
over the localization of a commutative ring $R$
at a multiplicatively closed subset $S$
that works under the following hypotheses:
the ring $R$ is coherent, i.e.,
we can compute finite generating sets of row syzygies of matrices over $R$,
and there is an algorithm
that decides for any given finitely generated ideal $I \subseteq R$
the existence of an element $r$
in $S \cap I$
and in the affirmative case computes $r$ as a concrete linear combination
of the generators of $I$.
\end{abstract}

\keywords{%
Computable ring, coherent strongly discrete ring, linear system, localization%
}
\subjclass[2010]{%
13B30, 
13P20
}
\maketitle


\section{Introduction}
The concept of rings equipped with algorithms
for dealing with linear systems
is fundamental
in constructive algebra \cite{CMS}.
A ring $R$ is called coherent if we have an algorithm
computing a finite generating set of the row syzygies of
a given matrix over $R$.
Moreover, $R$ is called
computable or coherent strongly discrete if we have
an algorithm for finding a particular solution of an inhomogeneous linear system over $R$.
Computable rings
provide the basis for an effective categorical framework for
homological algebra \cite{BL, PosFreyd, homalg-project}

In this paper we will address the following problem:
when is the localization $S^{-1}R$ of a coherent commutative ring $R$
at a multiplicatively closed subset $S \subseteq R$ computable?
With the investigation of this problem
we wish to contribute to the powerful framework
developed by Barakat and Lange-Hegermann in \cite{BL}
that renders the abelian category
of finitely presented $R$-modules constructive whenever a ring $R$
is known to be computable.

In \cite[Section 2.8.8]{GP}
it is shown how to solve linear equations
over $S_{>}^{-1}k[x_1, \dots, x_n]$, i.e.,
the polynomial ring
over a computable field $k$ in $n \in \N_0$ indeterminates
localized at the multiplicatively closed subset $S_{>}$
which consists of polynomials having leading monomial equal to $1$ for a given monomial ordering $>$.
In particular, this yields the computability of the localization 
$k[x_1, \dots, x_n]_{\langle x_{\ell+1}, \dots, x_n  \rangle} = S^{-1}_{>}k(x_1, \dots, x_{\ell})[ x_{\ell+1}, \dots, x_n ]$
by choosing a local ordering on $x_{\ell+1}, \dots, x_n$ for $\ell \in \N_0$
(see \cite[Example 1.5.3.4]{GP}).

The computability 
of the localization at a finitely generated maximal ideal $\mathfrak{m} \subset R$
of a computable ring $R$
is established by Barakat and Lange-Hegermann in \cite[Section 4]{BL}. Their algorithm avoids the computation of standard bases
over a local ordering in the special case $R = k[x_1, \dots, x_n]$ and $\mathfrak{m} = \langle x_1, \dots, x_n \rangle$.

In this paper we describe a general method for solving linear systems over $S^{-1}R$
for those coherent rings $R$ and multiplicatively closed subsets $S \subseteq R$
that can be equipped with the following extra datum: an algorithm
that decides for $\ell \in \N_0$ and any given finitely generated ideal $I = \langle f_1, \dots, f_{\ell} \rangle \subseteq R$
the existence of an element
in $S \cap I$
and in the affirmative case computes $a_1, \dots, a_{\ell} \in R$
such that $\sum_{i=1}^{\ell} a_i f_i \in S \cap I$.
From this general method, we can deduce the computability of $R_{\mathfrak{p}}$
for the localization of a computable ring $R$ at a finitely generated prime ideal
$\mathfrak{p}$ (Corollary \ref{corollary:localizations_at_prime_ideals}).

The paper is structured as follows.
After introducing the notion of coherent and computable rings
in Section \ref{section:computable_rings} we discuss
linear systems over $S^{-1}R$ in Section \ref{section:solving}.
The idea of our method for solving such systems is a generalization of the following observation: given a matrix $A \in R^{m \times n}$
and a row $b \in R^{1 \times n}$,
the existence of a solution $x \in R^{1 \times m}$ of the linear system $xA = b$
is equivalent to $1 \in \domain_R( [b]_A )$,
where $\domain_R( [b]_A )$ denotes the annihilator of $b$ regarded
as an element in $\cokernel( R^{1 \times m} \stackrel{A}{\longrightarrow} R^{1 \times n} )$.
We generalize this criterion to the localized case: if $A$ and $b$ are considered over $S^{-1}R$,
then a solution $x$ exists over $S^{-1}R$ if and only if 
$\domain_R( [b]_A ) \cap S$ is inhabited, i.e.,
there is an element $r$ in this intersection 
(Lemma \ref{lemma:dom}).
In the affirmative case a solution over $S^{-1}R$ can be constructed 
from a concrete expression of $r$ as a linear combination of a special set of generators of $\domain_R( [b]_A )$.
This will yield the computability of $S^{-1}R$ (Theorem \ref{theorem:computability}).

In the last Section \ref{section:examples}
we give some examples for our method. In particular,
we can describe how to find particular solutions of inhomogeneous
linear systems over 
\[
 (k[x_1, \dots, x_n]/I)_{\mathfrak{p}}
\]
\emph{without} the usage of Mora's tangent cone algorithm,
where $I$ is an ideal of $k[x_1, \dots, x_n]$ and
$\mathfrak{p}$ is a prime ideal of $k[x_1, \dots, x_n]/I$.

\section{Computable rings}\label{section:computable_rings}

In this paper $R$ will always denote a commutative unital ring.
Recall that $R$ is called \textbf{coherent} if
it comes equipped with an algorithm for computing syzygies:
\begin{enumerate}
 \item Given a matrix $A \in R^{m \times n}$ for $m,n \in \N_0$,
       we can find an  $o \in \N_0$ and a matrix $L \in R^{o \times m}$
       such that $LA = 0$. Furthermore, $L$ is universal with this property
       in the sense that for every other $p \in \N_0$ and matrix $T \in R^{p \times m}$
       such that $TA = 0$, there exists a $U \in R^{p \times o}$ such that
       $UL = T$.\label{enumitem1}
\end{enumerate}

\begin{remark}\label{remark:existential_quantifier}
 Following a fully constructive reading of this definition,
 the quantifier claiming the existence of $U$ for given $A$ and $T$
 also has to be realized algorithmically.
 We will come back to this very important point
 in Remark \ref{remark:come_back} and Remark \ref{remark:weak_kernel_lift}.
\end{remark}

Following \cite{BL} we call a coherent ring $R$ \textbf{computable}
if it additionally comes equipped with an algorithm for computing lifts:
\begin{enumerate}
  \setcounter{enumi}{1}
  \item Given two matrices $A \in R^{m \times n}$ and $B \in R^{q \times n}$,
        we can decide whether there exists a matrix $X \in R^{q \times m}$
        such that $XA = B$, and in the affirmative case construct such an $X$.
        We call $X$ a \textbf{lift of $B$ along $A$}.\label{enumitem2}
\end{enumerate}
In constructive algebra computable rings are also known as \textbf{coherent strongly discrete rings} \cite{CMS}.

We will refer to statement \eqref{enumitem1} as the \textbf{syzygy problem for $R$}.
Statement \eqref{enumitem2} is called the \textbf{lifting problem for $R$},
since it can be nicely rephrased as follows:
if we interpret the matrices $A$ and $B$ as $R$-module homomorphisms between free modules
$R^{1 \times m} \rightarrow R^{1 \times n}$ and $R^{1 \times q} \rightarrow R^{1 \times n}$, respectively, then we ask
whether the diagram 
\begin{center}
    \begin{tikzpicture}[label/.style={postaction={
          decorate,
          decoration={markings, mark=at position .5 with \node #1;}},
          mylabel/.style={thick, draw=none, align=center, minimum width=0.5cm, minimum height=0.5cm,fill=white}}]
          \coordinate (r) at (3.5,0);
          \coordinate (u) at (0,2);
          \node (n) {$R^{1 \times q}$};
          \node (m) at ($(n) - (u)$) {$R^{1 \times n}$};
          \node (q) at ($(m) + (r)$) {$R^{1 \times m}$};
          \draw[->,thick] (n) --node[left]{$B$} (m);
          \draw[->,thick] (q) --node[above]{$A$} (m);
    \end{tikzpicture}
  \end{center}
admits a lift, i.e., if there exists a module homomorphism $X$
making the diagram
\begin{center}
    \begin{tikzpicture}[label/.style={postaction={
          decorate,
          decoration={markings, mark=at position .5 with \node #1;}},
          mylabel/.style={thick, draw=none, align=center, minimum width=0.5cm, minimum height=0.5cm,fill=white}}]
          \coordinate (r) at (3.5,0);
          \coordinate (u) at (0,2);
          \node (n) {$R^{1 \times q}$};
          \node (m) at ($(n) - (u)$) {$R^{1 \times n}$};
          \node (q) at ($(m) + (r)$) {$R^{1 \times m}$};
          \draw[->,thick] (n) --node[left]{$B$} (m);
          \draw[->,thick] (q) --node[above]{$A$} (m);
          \draw[->,thick,dashed] (n) --node[above]{$X$} (q);
    \end{tikzpicture}
  \end{center}
commutative, and in the affirmative case we ask for a specific instance of such an $X$.

\begin{remark}\label{remark:come_back}
 Turning the existential quantifier in the definition of a coherent ring
 algorithmic as proposed in Remark \ref{remark:existential_quantifier}
 can be seen as a special instance of the lifting problem,
 namely finding a solution of
 \[
  XL = T.
 \]
 In actual implementations of computable rings it is advisable
 to have separate algorithms for the special and
 the general lifting problem,
 since in the special case we can benefit from additional knowledge.
 For example, if we deal with matrices over the polynomial ring,
 the syzygy matrix $L$ could have been already computed as a Gröbner basis with respect to 
 the induced ordering, and this knowledge turns the special lifting problem
 into a simple reduction
 (see Remark \ref{remark:weak_kernel_lift} for another example).
 Note that at the moment, there is no interface
 in the $\mathtt{homalg}$-project \cite{homalg-project} for the special lifting problem in computable rings\footnote{ 
 However, for a fixed ordering $<$ in the current session, 
 $\mathtt{homalg}$ matrices over polynomial rings that are created
 as Gröbner bases w.r.t.\ $<$ store this knowledge, and utilize it whenever appropriate. }.
 
 From a categorical point of view,
 being coherent for a ring means
 the existence of so-called weak kernels in the
 category of row modules (see \cite{PosFreyd}). In this light
 having a complete set of algorithms dealing with weak kernels,
 including their weak kernel lifts,
 appears very natural.
\end{remark}

\section{Solving linear systems over localizations of rings}\label{section:solving}

Let $R$ be a unital commutative ring equipped with a multiplicatively closed subset $S \subseteq R$,
i.e., $1 \in S$ and $r,s \in S$ implies $rs \in S$.
In this section we investigate the computability of the localization $S^{-1}R$
of $R$ at $S$.

Elements in $R$ give rise to elements in the localization via the not necessarily injective natural map
$R \rightarrow S^{-1}R: r \mapsto \frac{r}{1}$.
When we write $\frac{A}{d} \in S^{-1}R^{m \times n}$
we mean that $A$ is a matrix in $R^{m \times n}$,
$d \in S$, and $\frac{A}{d} = \left(\frac{A_{ij}}{d}\right)_{ij}$.
Note that every matrix with entries in $S^{-1}R$ has such
a representation by choosing a common denominator and representatives in $R$.

The solvability of the syzygy problem for $S^{-1}R$ is easy.

\begin{lemma}\label{lemma:syzygy_problem}
 If we can solve the syzygy problem for $R$,
 then the same is true for $S^{-1}R$.
\end{lemma}
\begin{proof}\hspace{-0.5em}\footnote{
 This proof can also be found in \cite[Lemma 4.3]{BL}. It had a typo in the exact sequence that had been communicated to the authors and now is fixed (v5).}
 Let $\frac{A}{d} \in S^{-1}R^{m \times n}$ and let
 $L \in R^{o \times m}$ be a solution of the syzygy problem of $A \in R^{m \times n}$.
 This is equivalent to $R^{1 \times o} \stackrel{L}{\longrightarrow} R^{1 \times m} \stackrel{A}{\longrightarrow} R^{1 \times n}$ being an exact sequence.
 Applying the exact localization functor $(\frac{-}{1})$ yields an exact sequence
 which proves that $\frac{L}{1}$ is a solution of the syzygy problem of $\frac{A}{1}$
 and thus, by the invertibility of $d$, also of $\frac{A}{d}$.
\end{proof}

\begin{remark}\label{remark:weak_kernel_lift}
 The proof of Lemma \ref{lemma:syzygy_problem} actually hides how we can find
 for given $\frac{T}{d'} \in S^{-1}R^{p \times m}$
 with $\frac{T}{d'} \cdot \frac{A}{1} = 0$
 a matrix $\frac{U}{d''} \in S^{-1}R^{p \times o}$
 such that
 \[
  \frac{U}{d''} \cdot \frac{L}{1} = \frac{T}{d'},
 \]
 so, we will briefly explain how it can be done.
 First, since $\frac{T}{d'} \cdot \frac{A}{1} = 0$ we can find an $s \in S$ such that $sTA = 0$.
 Next, since $L$ consists of row syzygies of $A$,
 we have a $U \in R^{p \times o}$ such that $UL = sT$.
 It follows that
 \[
  \left(\frac{U}{d's}\right) \cdot \frac{L}{1} = \frac{T}{d'}.
 \]
 Thus, we could quite easily establish a matrix $\frac{U}{d's}$ that solves a special
 instance of the lifting problem.
\end{remark}

Now, we turn to the general lifting problem for $S^{-1}R$.
Let $\frac{A}{d_A} \in S^{-1}R^{m \times n}, \frac{B}{d_B} \in S^{-1}R^{q \times n}$,
and $\frac{X}{d} \in S^{-1}R^{q \times m}$ for $m,n,q \in \N_0$.
Since
\[
 \frac{X}{d} \cdot \frac{A}{d_A} = \frac{B}{d_B} ~\Longleftrightarrow~ \frac{d_A d_B X}{d} \cdot \frac{A}{1} = \frac{B}{1},
\]
finding a lift of $\frac{B}{d_B}$ along $\frac{A}{d_A}$
is equivalent to finding a lift of $\frac{B}{1}$ along $\frac{A}{1}$,
which in turn is equivalent to finding lifts $\frac{X_i}{d_i} \in S^{-1}R^{1 \times m}$
for all rows of $\frac{B}{1}$ along $\frac{A}{1}$, i.e., $\frac{X_i}{d_i} \cdot \frac{A}{1} = \frac{B_{i,-}}{1}$ for $i = 1, \dots, q$.
Thus, it suffices to deal with the case $q = 1$ in which $B$ is a row vector (from now on we will call it $b$)
and the diagram of the simplified lifting problem is given by
\begin{center}
    \begin{tikzpicture}[baseline = (base),label/.style={postaction={
          decorate,
          decoration={markings, mark=at position .5 with \node #1;},
          },
          mylabel/.style={thick, draw=none, align=center, minimum width=0.5cm, minimum height=0.5cm,fill=white}}]
          \coordinate (r) at (3.5,0);
          \coordinate (u) at (0,2);
          \node (n) {$S^{-1}R^{1 \times 1}$};
          \node (m) at ($(n) - (u)$) {$S^{-1}R^{1 \times n}$};
          \node (q) at ($(m) + (r)$) {$S^{-1}R^{1 \times m}$.};
          \node (base) at ($0.5*(n) + 0.5*(m)$) {};
          \draw[->,thick] (n) --node[left]{$\frac{b}{1}$} (m);
          \draw[->,thick] (q) --node[above]{$\frac{A}{1}$} (m);
    \end{tikzpicture}
\end{center}

The key to the lifting problem lies in the following definition.
\begin{definition}\label{definition:domain}
  For a matrix $A \in R^{m \times n}$ and a row $b \in R^{1 \times n}$ we set
   \[
    \domain_{R}( [b]_A ) := \{ r \in R ~|~ \exists x \in R^{1 \times m}: xA = rb \} \subseteq R.
   \]
  \end{definition}

Note that $\domain_{R}( [b]_A )$ can also be described as the annihilator
of $b$ regarded as an element in $\cokernel( R^{1 \times m} \stackrel{A}{\longrightarrow} R^{1 \times n} )$.
In particular, $\domain_{R}( [b]_A )$ is an ideal of $R$.

\begin{remark}\label{remark:liftingproblem_ann}
 Whether a lift of $b$ along $A$ exists can be read off from $\domain_{R}( [b]_A )$:
 \[
  \left(\exists x \in R^{1 \times m}: xA = b\right) ~\Longleftrightarrow~ 1 \in \domain_{R}( [b]_A ).
 \]
\end{remark}

The last remark generalizes to the localized case.

\begin{lemma}\label{lemma:dom}
  Given $A \in R^{m \times n}$ and $b \in R^{1 \times n}$.
  Then there exists a lift $\frac{x}{d} \in S^{-1}R^{1 \times m}$ such that $\frac{x}{d} \cdot \frac{A}{1} = \frac{b}{1}$
  if and only if there exists an element $r \in \domain_{R}( [b]_A ) \cap S$.
\end{lemma}
\begin{proof}
  By Remark \ref{remark:liftingproblem_ann} a lift of $\frac{b}{1}$ along $\frac{A}{1}$
  exists if and only if
  $1 \in \domain_{S^{-1}R}\left( [\frac{b}{1}]_{\frac{A}{1}} \right)$.
  From the exactness of the localization functor, 
  we get $\langle \frac{a}{1} \mid a \in \domain_R(b,A) \rangle_{S^{-1}R} = \domain_{S^{-1}R}\left( [\frac{b}{1}]_{\frac{A}{1}} \right)$.
  Now, the claim follows from the fact that 
  for any ideal $I \subseteq R$,
  we have
  $1 \in \langle \frac{a}{1} \mid a \in I \rangle_{S^{-1}R}$
  if and only if there exists an element $r \in I \cap S$. 
%
%
\end{proof}

We turn to the case where $R$ is a coherent ring.

\begin{construction}\label{construction:dom}
 If $R$ is a coherent ring,
 then $\domain_{R}( [b]_A )$ can be constructed as follows.
 First, we find an $o \in \N_0$ and a solution $L \in R^{o \times (m+1)}$ of the syzygy problem
 \[
 X
 \cdot
 \begin{pmatrix}
  b \\
  A
 \end{pmatrix}
 = 0.
\]
Next, we decompose this solution $L$ as follows:
\[
 L =
  \left(\begin{array}{@{}c|c@{}}
  r_1 & L_1 \\
  \vdots & \vdots \\
  r_{o} & L_{o}
  \end{array}\right),
\]
where $r_i \in R$ and $L_i \in R^{1 \times m}$ for $i = 1, \dots, o$.
Then it easily follows that
\[
 \domain_{R}( [b]_A ) = \langle r_1, \dots, r_o \rangle_R.
\]
\end{construction}
In particular, $\domain_{R}( [b]_A )$ is a finitely generated ideal.
We should not discard the $L_i$ after this computation,
since their true value lies in the construction of lifts:

\begin{lemma}\label{lemma:coherent_plus_membership_is_computable}
 A ring $R$ is computable if and only if
 \begin{enumerate}
  \item $R$ is coherent
  \item we can effectively decide $1 \in I$
        for any finitely generated ideal $I = \langle f_1, \dots, f_{\ell} \rangle_R \subseteq R$,
        i.e., construct a linear combination $a_1, \dots, a_{\ell} \in R: \sum_{i=1}^{\ell} a_i f_i = 1$
        or disprove its existence.
 \end{enumerate}
\end{lemma}
\begin{proof}
 The ``only if'' direction is trivial, so
 we prove the ``if'' direction.
 Using the notation of Construction \ref{construction:dom},
 a solution of
 \[
  x \cdot A = b
 \]
 is simply given by
 \[
 \left(-\sum_{i=1}^o a_i L_i\right) \cdot A = b,
 \]
 where we use a linear combination $\sum_{i=1}^o a_i r_i = 1$.
\end{proof}

Our strategy for proving computability of $S^{-1}R$
is to generalize Lemma \ref{lemma:coherent_plus_membership_is_computable}.
Instead of finding a linear combination of $1$, we need to be able
to find a linear combination of an element of $S$:
\begin{definition}
Deciding whether for a given ${\ell} \in \N_0$ and finitely generated ideal $I = \langle f_1, \dots, f_{\ell} \rangle_R \subseteq R$
there exists an element in $I \cap S$, and in the affirmative case
constructing elements $a_1, \dots, a_{\ell} \in R$ such that $\sum_{i=1}^{\ell} a_i f_i \in I \cap S$
is what we call the \textbf{localization problem for $R$ at $S$}.
\end{definition}

Assume we have an algorithm solving the localization problem for $R$ at $S$.
Then we may use it to decide whether there exists an element in $\domain_{R}( [b]_A ) \cap S$
which by Lemma \ref{lemma:dom} is the case if and only if a lift of $\frac{b}{1}$ along $\frac{A}{1}$ exists.
In the affirmative case our algorithm
gives us $a_1, \dots, a_o \in R$ such that
\[
 \sum_{i=1}^o a_i r_i \in \domain_{R}( [b]_A ) \cap S,
\]
where the $r_i$ are the generators described in Construction \ref{construction:dom}.
Now, we can benefit from the already computed $L_i$ in Construction \ref{construction:dom}:
\begin{align*}
 \sum_{i=1}^o a_i \pmatrow{r_i}{L_i} \cdot \pmatcol{b}{A} = 0 \hspace{1em}
 &\Longleftrightarrow \hspace{1em} \left(\sum_{i=1}^o a_i r_i\right)b  + \left(\sum_{i=1}^o a_i L_i\right)A = 0 
\end{align*}
which in turn gives us a concrete formula for our desired lift\footnote{
If $R$ is a computable ring, then for any given $r \in \domain_{R}( [b]_A ) \cap S$
we can solve the lifting problem $xA = rb$ in $R$. It follows that $\frac{x}{r}$ is a lift of $\frac{b}{1}$ along $\frac{A}{1}$.
However, this strategy does not benefit from the already computed $L_i$
and thus can lead to slower computations.
}:
\[
 \left(\frac{\sum_{i=1}^o a_i L_i}{-\sum_{i=1}^o a_i r_i}\right) \cdot \frac{A}{1} = \frac{b}{1}.
\]
We have proven our main theorem.

\begin{theorem}\label{theorem:computability}
 Let $R$ be a coherent ring and $S \subseteq R$ a multiplicatively closed subset.
 Then $S^{-1}R$ is a computable ring
 if we can algorithmically solve the localization problem for $R$ at $S$.
\end{theorem}

\section{Examples}\label{section:examples}

Of special importance in algebraic geometry
are localizations of rings at prime ideals.

\begin{corollary}\label{corollary:localizations_at_prime_ideals}
 Let $R$ be a computable ring with a finitely generated 
 prime ideal $\mathfrak{p} = \langle p_1, \dots, p_m \rangle \subseteq R$
 for $m \in \N_0$.
 Then $R_{\mathfrak{p}} = S^{-1}R$
 is a computable ring, where
 $S := R - \mathfrak{p}$.
\end{corollary}
\begin{proof}
 By Theorem \ref{theorem:computability}
 we need to show how to solve the localization problem of $R$ at $S$.
 Given ${\ell} \in \N_0$ and a finitely generated ideal $I = \langle f_1, \dots, f_{\ell} \rangle \subseteq R$,
 then
 \begin{align*}
  \exists r \in I \cap S &\Longleftrightarrow \exists r \in I - \mathfrak{p} \\
  &\Longleftrightarrow \exists i \in \{1, \dots, {\ell}\}: f_i \not\in \mathfrak{p}.
 \end{align*}
 So, all we have to do is to test whether $f_i \in \mathfrak{p}$,
 which is equivalent to solving
 \[
  X \cdot 
  \begin{pmatrix}
  p_1 \\
  \vdots \\
  p_m
  \end{pmatrix}
  =
  f_i.
 \]
\end{proof}

\begin{remark}
 Let $k$ be a computable field, $n \in \N_0$,
 $I \subseteq k[x_1, \dots, x_n]$ an ideal.
 We set $R := k[x_1, \dots, x_n]/I$.
 Then 
 Corollary \ref{corollary:localizations_at_prime_ideals}
 gives us an algorithm to solve linear systems over
 \[
 \left(k[x_1, \dots, x_n]/I\right)_{\mathfrak{p}}
 \]
 for prime ideals $\mathfrak{p} \subseteq R$,
 since $R$ is a computable ring by means of Gröbner bases.
 In particular, we do not need the computation of a standard basis
 over a \emph{local} monomial ordering by means
 of the tangent cone algorithm \cite{MoraLocalAlgebra}.
\end{remark}

Since the localization of a polynomial ring at a prime ideal is a very interesting special case for computer algebra,
we discuss it at length in the following construction.

\begin{construction}\label{construction:polynomial_ring}
 Let $k$ be a computable field and $R := k[ x_1, \dots, x_n ]$ the polynomial ring in $n \in \N_0$
 indeterminates. 
 Let $\mathfrak{p} \subset R$ be a prime ideal with generators $p_1, \dots, p_m$ for $m \in \N_0$.
 Given a linear system 
 \[
  X \cdot \frac{A}{1} = \frac{b}{1}
 \]
 over $R_{\mathfrak{p}}$,
 where
 $\frac{A}{1} \in R_{\mathfrak{p}}^{m \times n}$ and $\frac{b}{1} \in R_{\mathfrak{p}}^{1 \times n}$ for $m,n \in \N_0$,
 we can find a solution (or disprove its existence) as follows:
 \begin{enumerate}
  \item Find a solution $L \in R^{o \times (m+1)}$ of the syzygy problem
 \[
 X
 \cdot
 \begin{pmatrix}
  b \\
  A
 \end{pmatrix}
 = 0.
\]
 This can be done with Gröbner basis techniques \cite[Algorithm 2.5.4]{GP},
 e.g., by computing a Gröbner basis of the rows of
 \[
  \left(\begin{array}{@{}c|c@{}}
  b & \text{\multirow{2}{*}{$I_{m+1}$}} \\
  A &
  \end{array}\right)
\]
 with a monomial ordering giving priority to the components in the left block.
 \item For $i=1, \dots, o$, let $\pmatrow{r_i}{L_i}$ denote the $i$-th row of $L$,
 where $r_i \in R$, $L_i \in R^{1 \times m}$.
 We check if $r_i \in \mathfrak{p}$
 with an algorithm\footnote{Such an algorithm needs
 a Gröbner basis of $\mathfrak{p}$. So, if we need to solve many different
 linear systems over the same ring $R_{\mathfrak{p}}$, then the determination
 of such a Gröbner basis can be seen as a preprocessing step.} for ideal membership \cite[Section 1.8.1]{GP}.
 The first $i$ such that $r_i \notin \mathfrak{p}$
 gives us the desired solution 
 \[
 \left(-\frac{L_i}{r_i}\right)
 \cdot
 \frac{A}{1}
 = \frac{b}{1}.
\]
 If there is no such $i$, then we successfully disproved the existence of a solution.
 \end{enumerate}
\end{construction}

The author believes that it is worth to implement
Construction \ref{construction:polynomial_ring}
for two reasons:
first, it covers localizations at all prime ideals (with given finite set of generators)
and therefore generalizes the work of Barakat and Lange-Hegermann in \cite{BL},
in which computability of a ring is sufficient to provide
a whole framework for effective homological algebra.
Second, it is not a priori clear how Construction \ref{construction:polynomial_ring}
performs in practice compared to the algorithm described in \cite[Proposition 4.5]{BL}:

\begin{proposition}[Barakat, Lange-Hegermann]
 Let $R$ be a computable ring with a maximal ideal $\mathfrak{m} = \langle m_1, \dots, m_{\ell} \rangle$ for ${\ell} \in \N_0$.
 A linear system 
 \[
  X \cdot \frac{A}{1} + \frac{b}{1} = 0
 \]
 over $R_{\mathfrak{m}}$,
 where
 $\frac{A}{1} \in R_{\mathfrak{m}}^{m \times n}$ and $\frac{b}{1} \in R_{\mathfrak{m}}^{1 \times n}$ for $m,n \in \N_0$,
 has a solution if and only if 
 the following linear system over $R$ has a solution:
 \begin{equation*}\label{equation:barakat_lange-hegermann}
   X \cdot
  \left(\begin{array}{@{}c@{}}
  A \\
  \left( m_1, \dots, m_{\ell} \right)^{\tr} \cdot b
  \end{array}\right)
  + b = 0.
 \end{equation*}
\end{proposition}
In the case $R = k[x_1, \dots, x_n]$ the costly part in solving this linear system
involves the computation of a Gröbner basis of
\[
 \left(\begin{array}{@{}c@{}}
  A \\
  \left( m_1, \dots, m_{\ell} \right)^{\tr} \cdot b
  \end{array}\right),
\]
whereas the costly part in Construction \ref{construction:polynomial_ring}
is the computation of syzygies. Testing Construction \ref{construction:polynomial_ring}
for the applications described in \cite{BL} would be an interesting project.

We end this section with two more examples.
\begin{example}
 Let $R$ be a computable commutative ring, and let $L = \langle h_1, \dots, h_m \rangle_R \subseteq R$
 be an ideal for $m \in \N_0$.
 Then $S := 1 + L$ is a multiplicatively closed set, and $S^{-1}R$
 the Zariskification of $A$ at $L$.
 For any finitely generated ideal $I = \langle f_1, \dots, f_{\ell} \rangle_R \subseteq R$
 with ${\ell} \in \N_0$,
 we have
 \begin{align*}
  \exists r \in S \cap I  \Longleftrightarrow \exists r_1, \dots, r_m, r'_1, \dots, r'_{\ell} \in R: 1 = \sum_{i=1}^m r_ih_i + \sum_{i=1}^{\ell} r'_i f_i.
 \end{align*}
 Since $R$ is computable, we can effectively solve this equation.
 Thus, Theorem \ref{theorem:computability} implies computability of the Zariskification.
\end{example}

\begin{example}
 Let $R$ be a computable commutative ring. For a polynomial $p = \sum_{i=0}^d a_it^i \in R[t]$
 where $d \in \N_0$, $a_i \in R$ for $i = 0, \dots d$, and $a_d \not= 0$ we define the leading term as
 $\LT( p ) = a_dt^d$.
 Such a polynomial is called monic if $a_d = 1$.
 Now, consider the multiplicatively closed subset
 \[
  S := \{ p \in R[t] \mid p \text{~is monic} \} \subseteq R[t]
 \]
 and the localization $R(t) := S^{-1}R[t]$.
 Let $I \subseteq R[t]$ be a finitely generated ideal
 with standard basis $G = \{g_1, \dots, g_n\} \subseteq I$ for $n \in \N_0$,
 i.e.,
 \[
  \langle \LT( g ) \mid g \in G \rangle_{R[t]} = \langle \LT( f ) \mid f \in I \rangle_{R[t]} \eqqcolon \LT(I).
 \]
 We write $\LT( g_i ) = c_i \cdot t^{d_i}$ for $c_i \in R$ and $d_i \in \N_0$.
 Then
 \begin{align*}
  \exists r \in I \cap S &\Longleftrightarrow \exists m \in \N_0: t^m \in \LT(I) \\
  &\Longleftrightarrow \exists m \in \N_0 ~ \exists a_i \in R: \sum_{i \in \{ 1 \dots n \mid m - d_i \geq 0\}} (a_i t^{m-d_i})(c_i t^{d_i}) = t^m \\
  &\Longleftrightarrow \exists a_i \in R: \sum_{i=1}^n a_i c_i = 1
 \end{align*}
 Since $R$ is computable, we can solve this last equation
 and if $(a_1, \dots, a_n)$ is such a solution,
 then
 \[
  \left(\sum_{i = 1}^n a_i t^{M-d_i}g_i\right) \in S \cap I
 \]
 where $M \geq \max\{d_1, \dots, d_n\}$.
Thus, whenever we can compute a standard basis of finitely generated ideals in $R[t]$ (e.g., if $R = \Z$),
then $R(t)$ is a computable ring.
See \cite[Chapter 4]{AL} for details on Gröbner bases over polynomial rings with coefficients
in a commutative noetherian ring.
\end{example}

\section{Outlook}

Computations within the algorithmic model of
the abelian category of finitely presented modules
$S^{-1}R\fpmod$ over a localized ring $S^{-1}R$
as implemented in the $\mathtt{homalg}$-project
are capable of outperforming
equivalent methods based on Mora's algorithm (see \cite[Section 6]{BL}).
However, the implementation in $\mathtt{homalg}$ is limited
to the case of localizations of computable rings $R$ at
maximal ideals $\mathfrak{m}$.
The methods described in this paper make it possible 
to model the categories $S^{-1}R\fpmod$ on the computer
for rings beyond $R_{\mathfrak{m}}$.
Their implementation is planned
within \CapPkg \cite{CAP-project},
a software project facilitating the implementation of category theory
based constructions. For example, $S^{-1}R\fpmod$ can be categorically constructed as the so-called Freyd category of row modules
over $S^{-1}R$ (see \cite{PosFreyd}).

Yet another drawback of the implementation of $R_{\mathfrak{m}}\fpmod$ in $\mathtt{homalg}$
is the dependency on Mora's algorithm 
for the computation of Hilbert series (see \cite[Remark 4.8]{BL}).
However, at least for modules of finite length,
as they appear for example in computations of intersection multiplicities,
we can get rid of this dependency by using
a purely categorical description of the filtration of a module induced by $\mathfrak{m}$.

The idea of a category theory based alternative approach to localization in computer algebra
can even be taken one step further:
instead of localizing the ring $R$, we can localize the whole category $R\fpmod$
in the sense of Serre quotients. This localization is again algorithmic \cite{BL_GabrielMorphisms}
and provides a framework for the category of coherent sheaves on quasi-affine schemes,
a proper generalization of $S^{-1}R\fpmod$.

\def\cprime{$'$} \def\cprime{$'$} \def\cprime{$'$} \def\cprime{$'$}
  \def\cprime{$'$}
\providecommand{\bysame}{\leavevmode\hbox to3em{\hrulefill}\thinspace}
\providecommand{\MR}{\relax\ifhmode\unskip\space\fi MR }
\providecommand{\MRhref}[2]{%
  \href{http://www.ams.org/mathscinet-getitem?mr=#1}{#2}
}
\providecommand{\href}[2]{#2}


\begin{thebibliography}{{hom}17}

\bibitem[AL94]{AL}
William~W. Adams and Philippe Loustaunau, \emph{An introduction to {G}r\"obner
  bases}, Graduate Studies in Mathematics, American Mathematical Society, 1994.
  \MR{1287608 (95g:13025)}

\bibitem[BLH11]{BL}
Mohamed Barakat and Markus Lange-Hegermann, \emph{An axiomatic setup for
  algorithmic homological algebra and an alternative approach to localization},
  J.~Algebra Appl. \textbf{10} (2011), no.~2, 269--293,
  (\href{http://arxiv.org/abs/1003.1943}{\texttt{arXiv:1003.1943}}).
  \MR{2795737 (2012f:18022)}

\bibitem[BLH14]{BL_GabrielMorphisms}
Mohamed Barakat and Markus Lange-Hegermann, \emph{{G}abriel morphisms and the
  computability of {S}erre quotients with applications to coherent sheaves},
  (\href{http://arxiv.org/abs/1409.2028}{\texttt{arXiv:1409.2028}}), 2014.

\bibitem[CMS12]{CMS}
Thierry Coquand, Anders M{\"o}rtberg, and Vincent Siles, \emph{Coherent and
  strongly discrete rings in type theory}, pp.~273--288, Springer Berlin
  Heidelberg, Berlin, Heidelberg, 2012.

\bibitem[GP02]{GP}
G.~Greuel and G.~Pfister, \emph{A {\bf {s}ingular} introduction to commutative
  algebra}, Springer-Verlag, 2002, With contributions by Olaf Bachmann,
  Christoph Lossen and Hans Sch\"onemann. \MR{MR1930604 (2003k:13001)}

\bibitem[GSP17]{CAP-project}
Sebastian Gutsche, Øystein Skartsæterhagen, and Sebastian Posur, \emph{The
  $\mathtt{CAP}$ project -- {C}ategories, {A}lgorithms, and {P}rogramming},
  (\url{http://homalg-project.github.io/CAP_project}), 2013--2017.

\bibitem[{hom}17]{homalg-project}
{homalg~project~authors}, \emph{The $\mathtt{homalg}$ project -- {A}lgorithmic
  {H}omological {A}lgebra}, (\url{http://homalg-project.github.io}),
  2003--2017.

\bibitem[Mor91]{MoraLocalAlgebra}
Teo Mora, \emph{La queste del {S}aint {${\rm Gr}_a(AL)$}: a computational
  approach to local algebra}, Discrete Appl. Math. \textbf{33} (1991), no.~1-3,
  161--190, Applied algebra, algebraic algorithms, and error-correcting codes
  (Toulouse, 1989). \MR{1137744}

\bibitem[Pos17]{PosFreyd}
Sebastian Posur, \emph{{A constructive approach to Freyd categories}}, ArXiv
  e-prints (2017),
  (\href{https://arxiv.org/abs/1712.03492}{\texttt{arXiv:1712.03492}}).

\end{thebibliography}
\end{document}